\documentclass[a4paper,twoside,reqno]{amsart}
\usepackage{amssymb}
\usepackage{amsmath}
\usepackage[mathscr]{eucal}
\usepackage{hyperref}
\usepackage{graphicx}
\usepackage[all]{xy}
\usepackage{tikz}
\usepackage{tikz-cd} 
\usepackage{multirow}
\usepackage{booktabs}
\usepackage{float}
\usepackage{placeins}
\usepackage[capitalise]{cleveref}
\usepackage{enumitem}

\crefformat{equation}{equation~(#1)}

\theoremstyle{plain}
\newtheorem{prop}{Proposition}[section]
\newtheorem{thm}[prop]{Theorem}
\newtheorem{cor}[prop]{Corollary}
\newtheorem{lem}[prop]{Lemma}
\newtheorem{conj}[prop]{Conjecture}

\newtheorem*{thm*}{Theorem}

\theoremstyle{definition}
\newtheorem{dfn}[prop]{Definition}
\newtheorem{rem}[prop]{Remark}
\newtheorem{example}[prop]{Example}
\newtheorem{lab}[prop]{}

\newcommand{\To}{\Rightarrow}

\newcommand{\C}{{\mathbb{C}}}
\renewcommand{\P}{{\mathbb{P}}}
\newcommand{\R}{{\mathbb{R}}}
\newcommand{\N}{{\mathbb{N}}}
\newcommand{\Z}{{\mathbb{Z}}}

\DeclareMathOperator{\im}{im}

\DeclareMathOperator{\interior}{int}

\DeclareMathOperator{\rk}{rk}

\DeclareMathOperator{\gl}{GL}

\DeclareMathOperator{\gram}{Gram}

\DeclareMathOperator{\spn}{span}
\DeclareMathOperator{\codim}{codim}
\DeclareMathOperator{\initial}{in}
\DeclareMathOperator{\gin}{gin}

\newcommand{\du}{{\scriptscriptstyle\vee}}

\newcommand{\ol}{\overline}

\newcommand{\all}{\forall\,}

\renewcommand{\subset}{\subseteq}
\renewcommand{\supset}{\supseteq}
\newcommand{\bil}[2]{\langle{#1},{#2}\rangle}
\newcommand{\mm}{\mathfrak{m}}
\newcommand{\sz}{\mathcal{Z}}

\newcommand{\id}[1]{\langle #1 \rangle}

\newcommand{\bpf}{\bp-free}
\newcommand{\bp}{base-point}
\newcommand{\gs}{Gram spectrahedron}
\newcommand{\gsa}{Gram spectrahedra}
\newcommand{\cc}{change of coordinates}
\newcommand{\nsingular}{non\--singular}

\renewcommand{\emptyset}{\varnothing}

\renewcommand{\epsilon}{\varepsilon}
\renewcommand{\theta}{\vartheta}

\newcommand\mapsfrom{\mathrel{\reflectbox{\ensuremath{\mapsto}}}}
\newcommand{\todfn}[1]{\textit{#1}}

\newcommand{\ra}[1]{\renewcommand{\arraystretch}{#1}}

\author{Julian Vill}
\title{dimensions of faces of gram spectrahedra}

\begin{document}

\begin{abstract}
Let $f\in\Sigma_{n,2d}$ be a sum of squares. The \gs\ of $f$ is a compact, convex set that parametrizes all sum of squares representations of $f$. Let $F\subset\gram(f)$ be a face of its \gs. We are interested in upper bounds for the dimension of $F$.
We show that this upper bound can be determined combinatorially. As it turns out, if the degree is large enough, a face realizing this bound, is a face of a \gs\ such that the form $f$ is singular. Thus we are also interested in finding better bounds whenever the form $f$ is smooth.
\end{abstract}

\maketitle

\section{Introduction}

We always assume that $d,\, n\ge 2$. We write $\R[x_1,\dots,x_n]_d$ for the vector space of homogeneous polynomials of degree $d$ and $\Sigma_{n,2d}\subset\R[x_1,\dots,x_n]_{2d}$ for the cone of sums of squares of degree $2d$. Let $f\in\Sigma_{n,2d}$ be a sum of squares (sos), e.g. there exist $p_1,\dots,p_r\in\R[x_1,\dots,x_n]_d$ such that $f=\sum_{i=1}^r p_i^2$. On the one hand one associates to this representation a Gram matrix $G$ such that $f=X^TGX=(p_1,\dots,p_r)(p_1,\dots,p_r)^T$ where $X$ is the vector containing all monomials of degree $d$, on the other hand, every positive semidefinite (psd) Gram matrix gives rise to a sum of squares representation of $f$, since every psd matrix $G$ can be decomposed as $G=H^TH$. Thus the set of all psd Gram matrices parametrizes the sum of squares representations of $f$ (up to orthogonal equivalence). The set of such matrices $G$ where $G$ is positive semidefinite is called the \gs\ $\gram(f)$ of $f$. It is by definition a spectrahedron, that is, the intersection of an affine-linear subspace with the cone of psd matrices. In this case the subspace is given by the linear conditions imposed on the matrix by the equation $f=X^TGX$.

Let $f\in\Sigma_{n,2d}$ and $\gram(f)$ its \gs. For any face $F$ of $\gram(f)$ any relative interior point has the same rank, hence it is well-defined to call this the rank of the face $F$.
For a Gram matrix the rank is the length of the corresponding sum of squares representation of $f$.
For some fixed rank $r$ we are interested in possible dimensions of a face of this rank (for fixed $n,d$).
It turns out that this is a purely algebraic question: Let $G$ be a Gram matrix in the relative interior of $F$ and write $U\subset\R[x_1,\dots,x_n]_d$ for the subspace spanned by the entries of $GX$. If $G$ corresponds to a representation $f=\sum_{i=1}^r p_i^2$, then $U$ is also given by the span of the $p_i$, hence especially $\dim U=\rk G$. 
Then the dimension of $F$ is given by the dimension of the kernel of the multiplication map $\mathsf{S}^2U\to \R[x_1,\dots,x_n],\, p\otimes q\mapsto pq$. Here $\mathsf{S}^2U$ is the second symmetric power of $U$, which we understand as all tensors in $U\otimes U$ that are invariant under permutation of the factors.
We are thus mainly interested in understanding the dimension of the image of this map which we denote by $U^2$, the subspace spanned by all product $pq$ where $p,\,q\in U$.

We write $R:=\R[x_1,\dots,x_n]$ and omit the $n$ in the notation. Furthermore it turns out to be more convenient to talk about the codimensions of $U$ and $U^2$ (in $R_d$ and $R_{2d}$ respectively).

In \cref{sec:Strongly stable subspaces} we determine upper bounds for $\codim U^2$ for fixed $n,d$ and $k$ where $k$ is the codimension of $U$. Although the bounds are tight, the subspaces realizing these bounds will usually have a \bp. This means that there exists a point in $\C^n$ such that every element in $U$ vanishes at that point. Especially these correspond to faces of \gsa\ of forms $f$ that are singular at some point.

We therefore consider subspaces that are \bpf\ to find better bounds for \nsingular\ sos forms, since this shows how \gsa\ look like for "generic" $f\in\Sigma_{n,2d}$.

In \cref{sec:Subspaces of codimension 1 and 2} we show that the bounds for $\codim U^2$ we determined in \cref{sec:Strongly stable subspaces} are far from optimal if the codimension of $U$ is 1 or 2 and $U$ is \bpf. In these cases we find uniform bounds for $\codim U^2$, not depending on $n$ or $d$, whereas the bounds in \cref{sec:Strongly stable subspaces} do always depend on $n$. Since we make use of some commutative algebra in this section as well as later on, we introduce Theorems of Macaulay, Gotzmann and Green in \cref{sec:Some commutative algebra}.

In \cref{sec:arbitrary codimension} we look at the same problem whenever $U\subset R_d$ is \bpf\ and has small codimension $k$ compared to $d$. The main result we show is that whenever $k\le d-1$, there is a uniform bound for $\codim U^2$, not depending on $n$ and $d$, as in the case of codimension $1$ and $2$. 
This is done via an inductive argument, where we reduce the degree and increase or decrease the number of variables, mainly using Green's Theorem. The degree reduction is done in \cref{sec:Reducing the degree}, the increasing of the variables is carried out in \cref{sec:Lifting subspaces} and the rest in done in \cref{sec:arbitrary codimension} itself.

\begin{thm*}[\cref{thm:independent_bound}]
Let $k\le d-1$ be a positive integer. Then for every $n\ge 2$ and every \bpf\ subspace $U\subset \R[x_1,\dots,x_n]_d$ of codimension $k$ we have 
\[
\codim U^2\le C(k),
\]
where $C(k)$ is a constant independent of $n$ and $d$.
\end{thm*}

We will see that $C(k)$ can be calculated combinatorially for any fixed $k$. It seems likely that it grows approximately as $k^3$, however as discussed in \cref{rem:growth_bound}, we can not proof this. This could possibly be done in some future work.

\section{Preliminaries}

For handling \gsa\ we will use the coordinate-free approach of \cite{scheiderer2018}. To this end we summarize the most important parts here.

Let $R:=\R[x_1,\dots,x_n]$ and let $\mathsf{S}^2R$ be the algebra of symmetric tensors. Let $\theta\in \mathsf{S}^2R_d$ and write $\theta=\sum_{i=1}^r p_i\otimes q_i$, then $\theta$ defines a symmetric bilinear map $b_\theta\colon R_d^\du\times R_d^\du\to \R$ via $b_\theta(\lambda,\mu)\mapsto\sum_{i=1}^r \lambda(p_i)\mu(q_i)$. We say that $\theta$ is psd if $b_\theta(\lambda,\lambda)\ge 0$ for all $\lambda\in R_d^\du$. Choose any representation $\theta=\sum_{i=1}^r p_i\otimes q_i$ where the sets $(p_i)_i$ and $(q_i)_i$ are linearly independent, then $\im(\theta):=\spn(p_i\colon i=1,\dots,r)=\spn(q_i\colon i=1,\dots,r)$. One possibility is to write $\theta=\sum_{i=1}^r \pm p_i\otimes p_i$ for linearly independent forms $p_i$, this follows immediately from the fact that symmetric matrices can be diagonalized. We also set $\rk(\theta):=\dim \im(\theta)$, called the \todfn{rank} of $\theta$.

With this notation the Gram map is given by 
\[
\mu\colon \mathsf{S}^2R_d\to R_{2d},\, p\otimes q\mapsto pq,
\]
and the \todfn{\gs} of $f\in R_{2d}$ is defined as
\[
\gram(f)=\mu^{-1}(f)\cap S_+^2R_d
\]
where $S_+^2R_d$ denotes the cone of all positive semidefinite tensors in $\mathsf{S}^2R_d$. Fixing a basis of $R_d$ identifies elements in $S_+^2R_d$ with the cone of psd matrices of size $\binom{n-1+d}{d}$, every tensor in $\gram(f)$ with a Gram matrix and $\gram(f)$ as defined above with the usual \gs.

Let $f\in\Sigma_{n,2d}$ and let $F\subset \gram(f)$ be a non-empty face and let $\eta$ be a relative interior point of $F$. Denote by
\[
\mathcal U(F):=\sum_{\theta\in F} \im(\theta) = \im(\eta)
\]
the \todfn{subspace corresponding to $F$}. Since the rank of any interior point of $F$ is the same we can define the rank of $F$ as $\rk(F):=\rk(\eta)=\dim\,\mathcal U(F)$. On the other hand, for any subspace $U\subset R_d$ let
\[
\mathcal F(U):=\{\theta\in\gram(f)\colon \im(\theta)\subset U\}
\]
denote the \todfn{face corresponding to $U$} (on $\gram(f)$). We call a subspace $U\subset R_d$ \todfn{$f$-facial}, if there exists a basis $p_1,\dots,p_r$ of $U$ such that $f=\sum_{i=1}^r p_i^2$.

\begin{prop}[{\cite[Proposition 2.10.]{scheiderer2018}}]
Let $f\in\Sigma_{n,2d}$. Then there is a bijection between non-empty faces of its \gs\ $\gram(f)$ and $f$-facial subspaces of $R_d$ given by
\begin{align*}
F &\mapsto \mathcal U(F)\\
\mathcal F(U) &\mapsfrom U.
\end{align*}
\end{prop}

\begin{prop}[{\cite[Proposition 3.6.]{scheiderer2018}}]
Let $f\in\Sigma_{n,2d}$ and let $F\subset \gram(f)$ be a face with corresponding subspace $U:=\mathcal U(F)$ then
\[
\dim F=\dim \ker(\mu|_{\mathsf{S}^2U}\colon \mathsf{S}^2U\to U^2)
\]
where $U^2=\spn(pq\colon p,q\in U)$.
\end{prop}

Especially the dimension of the face can be calculated purely algebraically. Rewriting the dimension we get
\[
\dim F=\binom{\dim U+1}{2}-\dim R_{2d}+\codim U^2.
\]
If we are working with a fixed number of variables and a fixed degree, the dimension of a face therefore only depends on the dimension of $U$ and the dimension of $U^2$.

For the rest of the paper we thus focus on determining $\codim U^2$ for appropriate subspaces $U$. Since $\codim U^2$ is invariant under field extensions, we can assume that our ground field is algebraically closed, henceforth we will work over $\C$.
By $A$ we always denote $\C[x_1,\dots,x_n]$. If we want to emphasize the number of variables we also write $A(n)$.

Let $U\subset A_d$ be a subspace. For any subset $S\subset A$ we write $\sz(S):=\{\xi\in\P^{n-1}\colon f(\xi)=0\quad \all f\in S\}$ for the zero set of $S$ in $\P^{n-1}$ where $\P^{n-1}$ is the complex projective space. 
We say that $U$ is \todfn{\bpf} if $\sz(U)=\emptyset$.

If $\succeq$ is a monomial ordering and $I\subset A$ is a homogeneous ideal, let $\initial_\succeq(I)$ be the \todfn{initial ideal} of $I$. Whenever we talk about the initial ideal $\id{U}$, we also write $\initial_\succeq(U)$ instead of $\initial_\succeq(\id{U})$, where $\id{U}$ is the ideal generated by $U$ in $A$.

Most claims do hold for every monomial ordering or at least any elimination ordering, however for simplicity, if not explicitly stated otherwise, we always work with the lexicographic-ordering such that $x_1<\dots <x_n$, and also write $\initial(I)$ instead of $\initial_\succeq(I)$. 

Lastly it is necessary not only to talk about subspaces but also about their orthogonal complements. We will always do this wrt the apolarity pairing:

\begin{dfn}
For $i=1,\dots,n$ define the differential operator $\partial_i:=\frac{\partial}{\partial x_i}$ and $\partial:=(\partial_1,\dots,\partial_n)$, $\partial^\alpha=\partial^{\alpha_1}\cdot\dots\cdot\partial^{\alpha_n}$ for $\alpha\in\Z^n_+$. For $f=\sum_\alpha c_\alpha x^\alpha\in A$ define $f(\partial):=\sum_\alpha c_\alpha \partial^\alpha$. For every $m\ge 0$ we then have the following bilinear form on $A_m$:
\[
\bil{f}{g}:=\frac{1}{m!}f(\partial)(g)=\frac{1}{m!}g(\partial)(f),\ \all f,g\in A_m.
\]
This bilinear form is called the \todfn{apolarity pairing}. It is a perfect pairing and in the real case a scalar product. Furthermore for every $u\in\C^n$ and $l:=\sum_{i=1}^n u_ix_i\in A_1$ and $f\in A_m$ we have $\bil{l^m}{f}=f(u)$. (See for example \cite[Lemma 1.15]{ika1999})
\end{dfn}

\section{Strongly stable subspaces}
\label{sec:Strongly stable subspaces}

For fixed $n,d,k$ and subspaces $U\subset A_d$ of codimension $k$ we determine an upper bound for $\codim U^2$. The bound can always be realized by a monomial subspace, meaning that there exists a basis consisting of monomials, and this subspace can even be chosen to be strongly stable. Especially this bound is tight and can be computed combinatorially. Everything in this section is true for any monomial ordering, however as mentioned earlier we will fix the lex-ordering for convenience and omit it in the notation.

\begin{dfn}
Let $U\subset A_d$ be a monomial subspace. $U$ is called \todfn{strongly stable} if for every $1 \le i \le n$ the following holds:
\[
\all m\in U \text{ monomial}: \left(x_i|m \To \all j>i: x_j\frac{m}{x_i}\in U\right).
\]
\end{dfn}

Note that for every monomial ordering $\succeq$ where $x_1<\dots<x_n$ and every monomial $m$ such that $x_i|m$ we have $x_j\frac{m}{x_i}\succeq m$ for every $j>i$.

\begin{prop}[{\cite[Theorem 15.18]{eisenbud1995}}]
\label{prop:gin_existence}
Let $I\subset A$ be a homogeneous ideal. Then there exists a Zariski-open subset $V\subset \gl_n(\C)$ such that for every $G_1,G_2\in V$ the initial ideals satisfy $\initial(G_1I)=\initial(G_2I)$ where $G_iI:=\{p(G_i^{-1}x)\colon p\in I\}$ is the ideal $I$ is mapped to by the coordinate change $G_i$ for $i=1,2$.
\end{prop}

\begin{dfn}
Let $I\subset A$ be a homogeneous ideal and $G\in V$ as in \cref{prop:gin_existence} then
\[
\gin(I):=\initial(GI)
\]
is called the \todfn{generic initial ideal} of $I$.
\end{dfn}

\begin{prop}[{\cite[Theorem 15.20, 15.23]{eisenbud1995}}]
\label{prop:gin_is_stongly_stable}
Let $I\subset A$ be a homogeneous ideal, then for every $s \ge 0$ the vector space $\gin(I)_s$ is strongly stable.
\end{prop}

Since our goal is to determine $\codim U^2$ we will use the following notation.

\begin{dfn}
For $n,\,d\ge 2$ and $k\ge 1$ let
\[
m(n,d,k)=\max\{\codim U^2\colon U\subset A(n)_d\text{ subspace},\, \codim U=k\}.
\]
\end{dfn}

The next Proposition gives a combinatorial bound for this maximum.

\begin{prop}[{\cite[Proposition 2.2]{bc2018}}]
\label{prop:ss_is_minimal}
Let $k\in\N$. There exists a strongly stable subspace $U\subset A_d$ of codimension $k$ such that
\[
m(n,d,k) = \codim U^2.
\]
\end{prop}
\begin{proof}
For every subspace $U\subset A_d$ we have $\initial(U)^2\subset \initial(U^2)$. Note that the Hilbert function of $\initial(U^2)$ is the same as the Hilbert function of $\id{U^2}$. Therefore $\dim (\initial(U)^2)_{2d}\le \dim U^2$ and hence the minimal value of $\dim U^2$ is attained by a monomial subspace $U$. Applying a general \cc\ we can assume that $\initial(U)$ is the generic initial ideal. Hence it is strongly stable by \cref{prop:gin_is_stongly_stable}.
\end{proof}

\begin{rem}
\label{rem:example_ss}
We know that the largest codimension of $U^2$ is achieved by a strongly stable subspace $U$. For small $n,d$ and codimension $k$ this is a list of the largest codimension of $U^2$ such that $\codim U=k$ and $U\subset A_d$ for $n=2,3,4,5,6$. This has been calculated with SAGE by first finding all strongly stable subspaces of some fixed codimension and then finding the maximum of all $\codim U^2$.

\FloatBarrier
{\tiny
\begin{table*}[ht]\centering
\ra{1.3}
\begin{tabular}{@{}rrrrrrrrrrcrrrrrrrrr@{}}\toprule
&\multicolumn{9}{c}{$n=3$} & \phantom{abc} & \multicolumn{8}{c}{$n=4$}\\
\midrule
$d=$ && $2$ & $3$ & $4$ & $5$ & $6$ & $7$ & $8$ & $9$ && $2$ & $3$ & $4$ & $5$ & $6$ & $7$ & $8$ & $9$\\ \midrule
$\codim U=$\\
$1$ && $3$ & $3$ & $3$ & $3$ & $3$ & $3$ & $3$ & $3$ && $4$ & $4$ & $4$ & $4$ & $4$ & $4$ & $4$ & $4$\\
$2$ && $6$ & $6$ & $6$ & $6$ & $6$ & $6$ & $6$ & $6$ && $8$ & $8$ & $8$ & $8$ & $8$ & $8$ & $8$ & $8$\\
$3$ && $10$ & $10$ & $10$ & $10$ & $10$ & $10$ & $10$ & $10$ && $13$ & $13$ & $13$ & $13$ & $13$ & $13$ & $13$ & $13$\\
$4$ && $12$ & $13$ & $13$ & $13$ & $13$ & $13$ & $13$ & $13$ && $20$ & $20$ & $20$ & $20$ & $20$ & $20$ & $20$ & $20$\\
$5$ && $14$ & $16$ & $17$ & $16$ & $16$ & $16$ & $16$ & $16$ && $23$ & $24$ & $25$ & $24$ & $24$ & $24$ & $24$ & $24$\\
$6$ && $-$ & $21$ & $21$ & $21$ & $21$ & $21$ & $21$ & $21$ && $26$ & $29$ & $29$ & $31$ & $28$ & $28$ & $28$ & $28$\\
$7$ && $-$ & $23$ & $24$ & $24$ & $25$ & $24$ & $24$ & $24$ && $30$ & $35$ & $35$ & $35$ & $37$ & $35$ & $35$ & $35$\\
$8$ && $-$ & $25$ & $27$ & $27$ & $28$ & $29$ & $27$ & $27$ && $32$ & $39$ & $40$ & $41$ & $41$ & $43$ & $40$ & $40$\\
$9$ && $-$ & $27$ & $30$ & $31$ & $31$ & $32$ & $33$ & $31$ && $34$ & $45$ & $45$ & $45$ & $47$ & $47$ & $49$ & $45$\\
\bottomrule
\bottomrule
&\multicolumn{9}{c}{$n=5$} & \phantom{abc} & \multicolumn{8}{c}{$n=6$}\\
\midrule
$d=$ && $2$ & $3$ & $4$ & $5$ & $6$ & $7$ & $8$ & $9$ && $2$ & $3$ & $4$ & $5$ & $6$ & $7$ & $8$ & $9$\\ \midrule
$\codim U=$\\
$1$ && $5$ & $5$ & $5$ & $5$ & $5$ & $5$ & $5$ & $5$ && $6$ & $6$ & $6$ & $6$ & $6$ & $6$ & $6$ & $6$\\
$2$ && $10$ & $10$ & $10$ & $10$ & $10$ & $10$ & $10$ & $10$ && $12$ & $12$ & $12$ & $12$ & $12$ & $12$ & $12$ & $12$\\
$3$ && $17$ & $16$ & $16$ & $16$ & $16$ & $16$ & $16$ & $16$ && $21$ & $19$ & $19$ & $19$ & $19$ & $19$ & $19$ & $19$\\
$4$ && $24$ & $25$ & $24$ & $24$ & $24$ & $24$ & $24$ & $24$ && $28$ & $31$ & $28$ & $28$ & $28$ & $28$ & $28$ & $28$\\
$5$ && $35$ & $35$ & $35$ & $35$ & $35$ & $35$ & $35$ & $35$ && $40$ & $40$ & $41$ & $40$ & $40$ & $40$ & $40$ & $40$\\
$6$ && $39$ & $40$ & $40$ & $41$ & $40$ & $40$ & $40$ & $40$ && $56$ & $56$ & $56$ & $56$ & $56$ & $56$ & $56$ & $56$\\
$7$ && $43$ & $47$ & $45$ & $46$ & $49$ & $45$ & $45$ & $45$ && $61$ & $62$ & $62$ & $62$ & $62$ & $62$ & $62$ & $62$\\
$8$ && $48$ & $54$ & $55$ & $54$ & $54$ & $57$ & $54$ & $54$ && $66$ & $71$ & $68$ & $68$ & $68$ & $71$ & $68$ & $68$\\
$9$ && $55$ & $60$ & $60$ & $63$ & $61$ & $62$ & $65$ & $59$ && $73$ & $79$ & $81$ & $79$ & $79$ & $79$ & $81$ & $79$\\
\bottomrule
\end{tabular}
\end{table*}
}
Note that in any case this codimension (of $U^2$) is achieved by a strongly stable subspace and thus $U$ has a real \bp. This means that if we look at the corresponding faces these will always be faces of \gsa\ of forms that lie on the boundary of the psd cone. In the case $k=1$ we will see that a face of maximal dimension always is the whole \gs\ for some form on the boundary of the psd cone. On the other hand for $k=2, n=3, d=2$ the maximum of $\codim U^2=6$ can also be attained by subspaces $U$ that are \bpf.
\FloatBarrier

Furthermore we see that in these examples $m(n,d,k)=m(n,k,k)$ for any $d\ge k$, this is actually always true as we show later.
\end{rem}

\begin{rem}
Note that in general we have $\initial(U)^2\subsetneq\initial(U^2)$. Take for example $U=\spn(x_1^2+x_2^2)^\perp$ with $n\ge 3$, then $U$ is spanned by all monomials except for $x_1^2$ and $x_2^2$ and by the binomial $x_1^2-x_2^2$. Furthermore $\initial(U)_2$ is spanned by all monomials except for $x_1^2$. Now one easily checks that $\codim \initial(U^2)_{4}=\codim U^2=2$ and $\codim \initial(U)_2^2=n$.
\end{rem}

This establishes upper bounds for $\codim U^2$. The lower bound will be attained by a generic subspace $U$ of fixed dimension. As the next Proposition shows, this dimension is not always the expected one.

\begin{prop}[{\cite[Proposition 2.8.]{bc2018}}]
\label{prop:quaternarydegenerate}
Let $n=4, d=2$ and $U\subset A_2$ be a subspace of codimension 2. Then $\dim U^2\le 34<35=\min\{\dim A_4, \dim \mathsf{S}^2U\}$.
\end{prop}
\begin{proof}
The dimension of $U^2$ is maximal for generic $U$ since $U^2$ is the image of the linear map $\mathsf{S}^2U\to A_{2d}$ and thus generically has maximal rank. Hence we can assume that $U$ is generic, then $U$ is apolar to a 2-dimensional space $W$. Since $W$ is generic it contains a quadratic form $q$ of rank 4. After a \cc\ we can assume that $q=x_1^2+x_2^2+x_3^2+x_4^2$ and that $W=\spn(q,a_1x_1^2+a_2x_2^2+a_3x_3^2+a_4x_4^2)$. By construction $U$ contains all polynomials apolar to both of these forms. Since all monomials $x_ix_j$ with $i\neq j$ are apolar to $W$ they are contained in $U$. Thus there are two quadratic relations, namely $(x_1x_4)(x_2x_3)=(x_1x_2)(x_3x_4)$ and $(x_1x_3)(x_2x_4)=(x_1x_2)(x_3x_4)$. And thus the kernel of the map $\mathsf{S}^2U\to U^2$ has dimension at least two. Since $\dim \mathsf{S}^2U=36$ we see that $\dim U^2\le 34$.
\end{proof}

\begin{example}[ternary quartics]
\label{ex:ternaryquartics}
Consider the case $n=3, d=2$. For $f\in\interior(\Sigma_{n,2d})$ any interior point of $\gram(f)$ has rank 6 and the dimension of $\gram(f)$ is also 6.

Which dimensions can faces $F$ of rank $r\in\{3,4,5\}$ have? First consider $r=5$. Corresponding to such a face we have a subspace $U\subset A_2$ of codimension 1. Then $\codim U^2\le 3$ by \cref{rem:example_ss} and we get $\dim F=\binom{5+1}{2}-\dim A_4+\codim U^2\le 3$.

For rank 4 faces we have $\codim U^2\le 6$, resulting in $\dim F\le 1$. The same holds for faces of rank 3.

This especially means that there are never faces of dimension $4$ and $5$ of any rank.
\end{example}

For later reference we will now consider \bpf\ monomial subspaces $U$ and find bounds for $\codim U^2$.

\begin{lem}
\label{lem:monomial_codim_1}
Let $d\ge 2$ and let $U\subset A_{d}$ be a \bpf, monomial subspace of codimension 1. Then the following hold: 
\begin{enumerate}
\item If $d=2$ then $\codim U^2\in \{0,2\}$,
\item if $d\ge 3$ then $\codim U^2\in\{0,1\}$.
\end{enumerate}
\end{lem}
\begin{proof}
There are (up to permutation of variables) exactly two monomials in $A_{2d}$ that have only one decomposition into two monomials of degree $d$ namely $x_1^{2d}$ and $x_1^{2d-1}x_2$. Therefore the only subspaces $U$ we need to consider are the orthogonal complements of $x_1^d$ and $x_1^{d-1}x_2$. In the first case $U$ has a \bp\, in the second case the only monomial not contained in $U^2$ is $x_1^{2d-1}x_2$ if $d\ge 3$ and thus $\codim U^2=1$. If $d=2$ then $x_1^3x_2$ and $x_1x_2^3$ are not contained in $U^2$ and we have $\codim U^2=2$. In any other case $U^2=A_{2d}$.
\end{proof}

\begin{lem}
\label{lem:monomial_subspaces_of_codimension_2}
Let $U\subset A_d$ be a \bpf, monomial subspace of codimension 2. Then the following hold: 
\begin{enumerate}
\item If $d=2$ then $\codim U^2\le 6$,
\item if $d\in\{3,4\}$ then $\codim U^2\le 4$,
\item if $d\ge 5$ then $\codim U^2\le 2$.
\end{enumerate}
\end{lem}
\begin{proof}
For $n=2$ this follows from \cref{prop:ss_is_minimal}, so assume that $n\ge 3$. One easily checks that for $d=3,4$ we have $\codim U^2\le 4$ and that $=4$ does occur in both cases. 
Now let $d\ge 5$. Up to permutation of the variables there are five monomials of degree $2d$ that have two or less representations as product of two monomials of degree $d$, namely
\begin{align*}
x_1^{2d}&=x_1^dx_1^d\\
x_1^{2d-1}x_2&=x_1^d(x_1^{d-1}x_2)\\
x_1^{2d-2}x_2x_3&=x_1^d(x_1^{d-2}x_2x_3)=(x_1^{d-1}x_2)(x_1^{d-1}x_3)\\
x_1^{2d-2}x_2^2&=x_1^d(x_1^{d-2}x_2^2)=(x_1^{d-1}x_2)^2\\
x_1^{2d-3}x_2^3&=x_1^d(x_1^{d-3}x_2^3)=(x_1^{d-1}x_2)(x_1^{d-2}x_2^2).
\end{align*}
Now one easily checks that for any two monomials of degree $d$ we exclude from $U$ there are at most two of the 5 monomials above not contained in $U^2$ since we cannot exclude $x_1^d$ because $U$ is \bpf. For example take 
\[
U=\spn(x_1^{d-1}x_2,x_1^{d-2}x_2^2)^\perp
\]
then the second and fourth are not contained in $U^2$. Thus we get $\codim U^2\le 2$.

For $d=2$ let $W=U^\perp$ then there are five possibilities up to permutation of the variables for $W$, namely
\FloatBarrier
\begin{table*}[ht]\centering
\ra{1.3}
\begin{tabular}{lll}
$W$ 				& $(U^2)^\perp$ 														& $\codim U^2$\\
\midrule
$x_1^2,x_2^2$ 		& $x_1^3x_i, x_2^3x_i\ \all i$ 											& $2n$\\
$x_1^2, x_1x_2$ 	& $x_1^3x_i, x_2^3x_1\ \all i$											& $n+1$\\
$x_1^2, x_2x_3$ 	& $x_1^3x_i, x_2^3x_3, x_3^3x_2\ \all i$ 								& $n+2$\\
$x_1x_2, x_1x_3$ 	& $x_1^3x_2, x_2^3x_1, x_1^3x_3, x_3^3x_1, x_1x_2^2x_3, x_1x_2x_3^2$ 	& $6$\\
$x_1x_2, x_3x_4$ 	& $x_1^3x_2, x_2^3x_1, x_3^3x_4, x_4^3x_3$								& $4$
\end{tabular}
\end{table*}
\FloatBarrier
In the first three cases $U$ has a \bp.
\end{proof}

\section{Some commutative algebra}
\label{sec:Some commutative algebra}

For the following sections we need some knowledge about the Hilbert functions of ideals generated by subspaces. Let $h_I(t)$ be the Hilbert function of the homogeneous ideal $I$ which we define as $h_I(t)=\dim (A/I)_t,\ t\ge 0$. We introduce theorems of Macaulay and Gotzmann concerning Hilbert functions and Green's Hyperplane Restriction Theorem for later reference.

\begin{dfn}
Let $a,d\in\N$, then $a$ can be uniquely written in the form
\[
a=\binom{k(d)}{d}+\binom{k(d-1)}{d-1}+\dots+\binom{k(1)}{1},
\]
where $k(d)>k(d-1)>\dots>k(1)\ge 0$, called the \todfn{$d$-th Macaulay representation} of $a$ (see \cite[Lemma 4.2.6.]{bh1998}). For any integers $s,t\in\Z$ define
\[
a_{(d)}|^s_t:=\binom{k(d)+s}{d+t}+\binom{k(d-1)+s}{d-1+t}+\dots+\binom{k(1)+s}{1+t}.
\]
Furthermore for $a<b$ we define $\binom{a}{b}=0$.
\end{dfn}

\begin{thm}[Macaulay's Theorem, {\cite[Corollary C.7.]{ikl1999}}, {\cite[Theorem 4.2.10]{bh1998}}]
\label{thm:macaulay}
Let $I\subset A$ be a homogeneous ideal and let $H=(h_i)_{i\ge 0}$ be the Hilbert function of $I$. 
Then 
\begin{enumerate}
\item $h_{i+1}\le (h_i)_{(i)}|^{1}_{1}$ for every $i\ge 0$, and
\item if there exists $j\in\N$ such that $j\ge h_j$ then $h_i\ge h_{i+1}$ for every $i\ge j$.
\end{enumerate}
\end{thm}

\begin{thm}[Gotzmann's Persistence Theorem, {\cite[Corollary C.17.]{ikl1999}}, {\cite[Theorem 2.6]{ams2018}}]
\label{thm:gotzmannpersistence}
Let $d \ge 0$ be an integer and let $I$ be a homogeneous ideal that is generated in degrees at most $d$ $(I=\id{I_{\le d}})$. Denote by $H=(h_i)_{i\ge 0}$ the Hilbert function of $I$. 
If $h_{d+1}=(h_d)_{(d)}|^{1}_{1}$ then 
$h_{d+l}=(h_d)_{(d)}|^l_l$ for all $l\ge 1$.
\end{thm}

\begin{cor}
\label{cor:macaulay_gotzmann}
Let $U\subset A_d$ be a subspace of codimension $k\le d$ and $H=(h_i),\ h_i:=h_{\id{U}}(i)$ the Hilbert function of $\id{U}$ . Then
\begin{enumerate}
\item $h_{d+1}=\codim A_1U\le k$ and 
\item if $h_{d+1}=k$ then $h_{d+i}=\codim A_iU=k$ for all $i\ge 1$.
\end{enumerate}
In case (ii) $\sz(U)\neq\emptyset$ is finite.
\end{cor}
\begin{proof}
(i): We have $h_d=\dim A_d/U=\codim U=k\le d$ and thus by \cref{thm:macaulay} (i) we also have $\codim A_1U\le k$.

(ii): Since $h_d\le d$ we get the $d$-th Macaulay representation $h_d=\binom{d}{d}+\dots+\binom{d-h_d+1}{d-h_d+1}$ and thus also $h_{d+1}\le (h_d)_{(d)}|^{1}_{1}=h_d$. By \cref{thm:macaulay} (ii) we know that $h_d\le h_{d+1}$ hence the inequality is an equality and thus (ii) follows from \cref{thm:gotzmannpersistence}.

In (ii) the Hilbert polynomial is the constant polynomial $k$, hence $\sz(U)$ is non-empty and finite. 
\end{proof}

\begin{cor}
\label{cor:degree2d-1}
Let $U\subset A_d$ be a \bpf\ subspace with $\codim U=k\le d$. Then $h_{\id{U}}(2d-1)\le 1$. If $k<d$ then $h_{\id{U}}(2d-1)=0$.
\end{cor}
\begin{proof}
The Hilbert function of $\id{U}$ has to be smaller than $(\dots,k,k-1,k-2,\dots,1,0)$ (dimension dropping by at least 1 in every degree), this follows by induction using the last corollary. Now the claim is immediate.
\end{proof}

\begin{dfn}
Let $I\subset A$ be a homogeneous ideal and $p\in A_s$ for some $s\ge 1$. Define the \todfn{ideal quotient}
\[
(I:p):=\bigoplus_{l\ge 0} (I:p)_l
\]
where
\[
(I:p)_l:=\{q\in A_l\colon pq\in I\}\subset A_l
\]
for every $l\ge 0$. If $U\subset A_d$ is a subspace, we write $(U:p):=(\id{U}:p)_{d-s}\subset A_{d-s}$.
\end{dfn}

\begin{lab}
\label{lab:setup_green}
Consider the following setup: Let $I\subset A$ be a homogeneous ideal and $l\in A_1$ a linear form. We have the graded exact sequence
\[
0 \to A/(I:l)(-1) \to A/I \to A/\id{I,l} \to 0.
\]
Let $h_i  = \dim (A/I)_i$ and $c_i  = \dim (A/(I,l))_i$.
\end{lab}

In this situation we have the following theorem due to Green.

\begin{thm}[Green's Hyperplane Restriction Theorem, {\cite[Theorem 1]{green1989}}]
\label{thm:green_thm}
For any $d\ge 0$ and a generic linear form $l\in A_1$ we have
\[
c_d \le (h_d)_{(d)}|^{-1}_0.
\]
\end{thm}

This can either be understood as a bound for $\dim \id{I,l}_d$ or to understand how many elements in $I$ are divisible by a generic linear form $l$.

Notation-wise this means that if $h_d=\binom{k(d)}{d}+\binom{k(d-1)}{d-1}+\dots+\binom{k(1)}{1}$, then $c_d\le \binom{k(d)-1}{d}+\binom{k(d-1)-1}{d-1}+\dots+\binom{k(1)-1}{1}$.

\section{Subspaces of codimension 1 and 2}
\label{sec:Subspaces of codimension 1 and 2}

Again let $A:=\C[x_1,\dots,x_n]$. We consider subspaces $U$ of codimension 1 and 2 and show that there is a uniform bound for $\codim U^2$ not depending on $n$ or $d$. Furthermore we show \cref{thm:reduction_number_of_variables} which will be our main tool in the next sections to reduce the number of variables.

\begin{lem}
\label{lem:not_contained_in_Pn-3}
Let $U\subset A_d$ be a \bpf\ subspace and $W:=U^\perp$. If $\sz(W)$ is not contained in any linear variety of codimension 2, then there exists a \cc\ such that $\initial(U)_d$ is \bpf.
\end{lem}
\begin{proof}
By assumption there exist linearly independent linear forms $l_1,\dots,l_{n-1}\in A_1$ such that $l_1^d,\dots,l_{n-1}^d\in U$. After a \cc\ we can assume that $x_1^d,\dots,x_{n-1}^d\in U$. Since $x_1<x_2<\dots <x_n$, it holds that $x_n^d\ge x^\alpha$ for any $\alpha\in\Z^n_+,\ |\alpha|=d$. Since $U$ is \bpf, there exists a form in $U$ such that $x_n^d$ occurs in it. Hence $\initial(U)_d$ contains $x_1^d,\dots,x_n^d$ which shows that the subspace $\initial(U)_d$ is \bpf. 
\end{proof}

The first case we look at are subspaces $U\subset A_d$ of codimension 1.

\begin{lem}
\label{lem:basepoint}
If $U\subset A_d$ is a subspace of codimension 1 and $U$ has a \bp\ then $\codim U^2=n$. 
\end{lem}
\begin{proof}
We can apply a \cc\ such that $U^\perp=\spn(x_1^d)$, then $U$ is the subspace spanned by all monomials except $x_1^d$. Now we see that for every $1\le i\le n$ the monomial $x_1^{2d-1}x_i$ is not contained in $U^2$ and thus $\codim U^2=n$.
\end{proof}

\begin{prop}
\label{prop:codim1possibledim}
Let $d\ge 2$ and $U\subset A_d$ be a \bpf\ subspace of codimension 1. Then the following hold:
\begin{enumerate}
\item If $d\ge 3$ then $\codim U^2\in\{0,1\}$,
\item if $d=2$ then $\codim U^2\le 2$.
\end{enumerate} 
\end{prop}
\begin{proof}
Write $W:=U^\perp=\spn(q)$ for some $q\in A_d$. No hypersurface is contained in a linear variety of codimension 2, hence by \cref{lem:not_contained_in_Pn-3} we can apply a \cc\ such that the subspace $\initial(U)_d$ is \bpf. Then $\dim U^2=\dim \initial (U^2)_{2d}\ge\dim (\initial(U)^2)_{2d}$. From \cref{lem:monomial_codim_1} we get
\[
\codim (\initial(U)^2)_{2d}\in\{0,1\}, \text{ if } d\ge 3 \text{ and } \le 2 \text{ if } d=2.
\]
\end{proof}

\begin{rem}
\begin{enumerate}[leftmargin=0.6cm]
\item[]
\item With a different method one can also improve this in the case $d=2$. Namely $\codim U^2\in\{0,2\}$, which means that the case $\codim U^2=1$ is actually not possible.
\item \cref{prop:codim1possibledim} shows that if $d\ge 3$ then for any $f\in\interior(\Sigma_{n,2d})$ and $F\subset \gram(f)$ a face of rank $\dim A_d-1$ it holds that $\dim F=\binom{r+1}{2}-\dim A_{2d} +\epsilon$ with $\epsilon\in\{0,1\}$. This is in fact also true as long as $f\notin\partial P_{n,2d}$.
\end{enumerate}
\end{rem}

Now we turn to the codimension 2 case. We find a bound for $\codim U^2$ by reducing either to monomial subspaces or to subspaces of binary forms.

First we show how to reduce the number of variables. The idea of the proof is the following: If $U\subset A[x_{n+1}]_d=\C[x_1,\dots,x_{n+1}]_d$ is a subspace of the form $U=x_{n+1}A[x_{n+1}]_{d-1}\oplus U'$ with $U'\subset A_d$, then $U^2=x_{n+1}^2A[x_{n+1}]_{2d-2}\oplus x_{n+1}A_{d-1}U'\oplus (U')^2$. Then $\codim U^2=\codim (U')^2+\codim A_{d-1}U'$. If $U$ does not have this nice form we have to argue slightly more carefully using the same idea.

\begin{thm}
\label{thm:reduction_number_of_variables}
Let $U\subset A_d$ be a subspace of codimension $k$. If there exists $2\le m\le n\ (R:=A(m))$ such that $U':=U\cap R_d$ satisfies $\codim_{R_d} U'=k$, then
\[
\codim_{A_{2d}} U^2\le (n-m)\codim_{R_{2d-1}} U'R_{d-1}+\codim_{R_{2d}} (U')^2.
\]
\end{thm}
\begin{proof}
Let $\mm=\id{x_{m+1},\dots,x_n}\subset A_d$ and $S:=\sum_{i=m+1}^n x_iA_{d-1}$ and write 
\[
U=U'\oplus V \oplus W.
\]
with $U'\subset R_d,\ V\subset S$ and $W=\spn(p_i+q_i\colon i=1,\dots,s)$ where $p_i\in R_d$ and $q_i\in S$. Here the $q_i$ cannot be zero since otherwise $U\cap R_d$ would have the wrong dimension. By assumption $\codim_{R_d} U'=k$ and thus
\begin{equation}
\label{eq:eq1}
V \oplus \spn(q_1,\dots,q_s)=S.
\end{equation}
Calculating $U^2$ we get
\[
U^2=(U')^2+(V+W)^2+U'(V+W).
\]
Since we are working with the lex-ordering, any monomial containing any $x_i,\ i\ge m+1$ is bigger than any monomial in $R$.

Firstly fix any monomial $x^\alpha$ such that $\alpha\in\Z^{n}_+, |\alpha|=2d$ and $\sum_{j\ge m+1} \alpha_j\ge 2$ then there exist $\beta,\gamma\in\Z^{n}_+,|\beta|=|\gamma|=d$ and $x_i,x_j,\ i,j\ge m+1$ such that $x_i|x^\beta, x_j|x^\gamma$ and $x^\alpha=x^\beta x^\gamma$. Then we have $x^\beta + p_\beta,\, x^\gamma + p_\gamma\in V+W$ for some $p_\beta,p_\gamma\in R_d$. Hence
\[
x^\alpha=\initial((x^\beta + p_\beta)(x^\gamma + p_\gamma))\in\initial((V+ W)^2)_{2d}\subset\initial(U^2)_{2d}.
\]
Secondly we have
\[
\initial(U'(V+W))=\initial(U'S)\supset\initial\left(\bigoplus_{i=m+1}^n x_i (U'R_{d-1})\right) \text{ by \cref{eq:eq1}}.
\]
This shows that for every $i=m+1,\dots, n$
\[
\mm^2A_{d-2},\initial((U')^2),\initial(x_iU'R_{d-1})\subset\initial(U^2).
\]
Counting dimensions we see that 
\[
\codim_{A_{2d}}\initial(U^2)\le (n-m)\codim_{R_{2d-1}} U'R_{d-1}+\codim_{R_{2d}} (U')^2.
\]
\end{proof}

\begin{rem}
The bound is sharp whenever $U=(x_{m+1},\dots,x_n)A_{d-1}\oplus U'$ as can be seen from the comment above \cref{thm:reduction_number_of_variables}.
\end{rem}

\begin{cor}
\label{cor:reduction_number_of_variables}
If the subspaces $U,U'$ in \cref{thm:reduction_number_of_variables} are \bpf\ and $k\le d-1$ then
\[
\codim U^2\le \codim (U')^2.
\]
\end{cor}
\begin{proof}
By \cref{cor:degree2d-1} the degree $2d-1$ component of $R/\id{U'}$ has dimension 0. Therefore the result follows from \cref{thm:reduction_number_of_variables}. 
\end{proof}

\begin{thm}
\label{thm:codimension2bounds}
Let $U\subset A_d$ be a \bpf\ subspace of codimension 2. Then the following hold:
\begin{enumerate}
\item If $d=2$ then $\codim U^2\le 6$,
\item if $d\ge 3$ then $\codim U^2\le 4$.
\end{enumerate}
For $d\le 4$ the bounds are tight.
\end{thm}
\begin{proof}
Let $W=U^\perp$. If $\sz(W)\neq\sz(l,l')$ for any two linear forms $l,l'\in A_1$ then the claim follows from \cref{lem:not_contained_in_Pn-3} and \cref{lem:monomial_subspaces_of_codimension_2}.

Now assume that $\sz(W)=\sz(x_1,x_2)$ and thus $x_3^d,\dots,x_n^d\in U,\ x_1^d,x_2^d\notin U$. 
Hence we can write
\[
U=\spn(x^\alpha+\mu_\alpha x_1^d+\lambda_\alpha x_2^d\colon \exists i\ge 3\colon x_i | x^\alpha)\oplus U'
\]
where $U'\subset\C[x_1,x_2]_d$ is a subspace of codimension 2. We distinguish two cases:
\begin{itemize}
\item[(a)] For all $\alpha$ we have $\mu_\alpha=\lambda_\alpha=0$,
\item[(b)] there exists $\alpha$ such that $(\mu_\alpha,\lambda_\alpha)\neq (0,0)$.
\end{itemize}

(a): Here $U$ has the form
\[
U=\spn(x_3,\dots,x_n)A_{d-1}\oplus U'.
\]
If $d=2$ this case cannot appear since $\dim U'=1$ and thus $U$ has a \bp. Hence we can assume that $d \ge 3$. Since $U$ is \bpf\ it follows that $U'$ is \bpf\ as a subspace of $\C[x_1,x_2]_d$. 
Then $\codim U^2 \le \codim U'^2\le 4$ by \cref{cor:reduction_number_of_variables}.

(b): Fix $\alpha$ such that $(\mu_\alpha,\lambda_\alpha)\neq (0,0)$. Consider the subspace 
\[
V:=U'\oplus \spn(\mu_\alpha x_1^d+\lambda_\alpha x_2^d)\subset\C[x_1,x_2]_d.
\]
This subspace has codimension 1 and thus $V^\perp=\spn(h)$ for some $h\in\C[x_1,x_2]_d$. Especially there exists $l\in\C[x_1,x_2]_1$ such that $l^d\in V$, namely the one evaluating $h$ in one of its zeroes. This shows that $x^\alpha+a l^d\in U$ or $a l^d\in U$ for some $a\neq 0$. Write $x^\alpha=x_1^{\alpha_1}x_2^{\alpha_2}M$ with $M\in\C[x_3,\dots,x_n]$. 
Applying the \cc\ to $\C[x_1,x_2]$ that maps $l$ to $x_2$ shows that $Mg+x_2^d\in U$ with $g\in\C[x_1,x_2]$ the image of $x_1^{\alpha_1}x_2^{\alpha_2}$ under the \cc. Now take any monomial ordering such that $x_1>x_2>\dots >x_n$ and such that $x_2^d$ is greater than any monomial in $Mg$ (for example a block ordering on $\{x_1,x_2\}$ and $\{x_3,\dots,x_n\}$ with grlex on each block).
Wrt this ordering $\initial(U)_d$ is \bpf. Now we finish as earlier, $\codim U^2\le \codim \initial(U)_d^2$ and using \cref{lem:monomial_subspaces_of_codimension_2} we get the bounds we wanted.

The bounds are tight for $d\le 4$ by the proof of \cref{lem:monomial_subspaces_of_codimension_2}.
\end{proof}

\begin{rem}
As can be seen from the proof we actually showed slightly more than mentioned in the theorem: For $d\le 4$ the bound $\codim U^2\le 4$ is tight for binary forms and thus for any $n\ge 2$. 

One can show that for $d\ge 5$ and $n=2$, it holds that $\codim U^2\le 3$ for \bpf\ subspaces of codimension 2, hence improving the bound by 1. It is not clear however if $\codim U^2=3$ is possible for $d\ge 5$.
\end{rem}

\section{Reducing the degree}
\label{sec:Reducing the degree}

In this section we show that for $d\ge k$ the function $d\mapsto m(n,d,k)$ is non-increasing for every fixed $n,k$. By definition this is equivalent to showing that for every subspace $U\subset A_d$ of codimension $k$ there exists a subspace $V\subset A_{d-1}$ of codimension $k$ such that $\codim U^2\le \codim V^2$ whenever $d > k$.

For the next proofs let us recall that for any subspace $U\subset A_d$ we have the exact sequence in \ref{lab:setup_green}:
\[
0 \to A_{d-1}/(U:l) \to A_d/U \to A_d/\id{U,l}_d \to 0.
\]
Especially if $\id{U,l}_d=A_d$ it follows that $\codim (U:l) = \codim U$.

\begin{lem}
\label{lem:green_small_k}
Let $U\subset A_d$ be a subspace of codimension $k$ and $k\le d$, then for a generic linear form $l\in A_1$ we have $\id{U,l}_d=A_d$ and $\codim (U:l)=\codim U$.
\end{lem}
\begin{proof}
With the notation from \cref{lab:setup_green} with $I=\id{U}$ we have
\[
h_d=k=\sum_{i=0}^{k-1} \binom{d-i}{d-i}
\]
since $k\le d$.
Hence by Green's Theorem 
\[
c_d\le (h_d)_{(d)}\vert^{-1}_0 = \sum_{i=0}^{k-1} \binom{d-i-1}{d-i}=0
\]
which means $A_d/\id{U,l}_d=0$ and therefore the first claim follows.

The second one is immediate from the exact sequence above.
\end{proof}

\begin{thm}
\label{thm:deg_reduction}
Let $U\subset A_d$ be a subspace of codimension $k\le d$ and let $l\in A_1$ be a generic linear form. With $V:=(U:l)$ the following inequality holds
\[
\codim U^2 \le \codim UV.
\]
If furthermore $k\le d-1$ then
\[
\codim U^2 \le \codim V^2.
\]
\end{thm}
\begin{proof}
Since $l$ is generic and $k\le d$ it follows from \cref{lem:green_small_k} that $\codim V=\codim U$ and $\id{U,l}_d=A_d$. Furthermore we have 
\[
A_{2d}=(\id{U,l}_d)^2\subset \id{U^2,l}_{2d},
\] 
hence $\codim (U^2:l)=\codim U^2$ by the exact sequence in \ref{lab:setup_green}. Since $UV\subset (U^2:l)$ we have
\[
\codim U^2=\codim (U^2:l)\le \codim UV.
\]
Now we do the same for $UV$. If we show that $\id{V,l}_{d-1}=A_{d-1}$, then
\[
A_{2d-1}=\id{U,l}_d\id{V,l}_{d-1}\subset \id{UV,l}_{2d-1}.
\]
Thus $\codim (UV:l)=\codim UV$ and $V^2\subset (UV:l)$ which means $\codim UV\le \codim V^2$.

It is left to show that $\id{V,l}_{d-1}=A_{d-1}$. This is equivalent to showing that $\codim (V:l)=\codim V$. Since $((U:l):l)=(U:l^2)$ this again is equivalent to showing that $\codim (U:l^2)=\codim V=\codim U$ or $\id{U,l^2}_d=A_d$. Since $l\in A_1$ is generic we can also apply a generic \cc\ to $U$, hence assume that $\initial(U)=\gin(U)$ and $l=x_1$. Then
\[
\dim\id{U,x_1^2}_d=\dim\initial(\id{U,x_1^2})_d\ge \dim \id{\initial(U),x_1^2}_d.
\]
Here the first equality follows from the fact that any ideal and its initial ideal have the same Hilbert function, the second one is immediate since $\initial(\id{U,x_1^2})\supset\id{\initial(U),x_1^2}$.

It is therefore enough to show that $\id{\gin(U),x_1^2}_d=A_d$. Since $k\le d-1$ every monomial of degree $d$ not contained in $\gin(U)_d$ is divisible by $x_1^2$: The lowest power $s$ of $x_1$ such that $x_1^sM$ is not contained in $\gin(U)_d$ for a monomial $M$ of degree $d-s$ is 2. This is realized by $\gin(U)_d^\perp=\spn(x_1^d,x_1^{d-1}x_2,\dots,x_1^2x_2^{d-2})$ and $k=d-1$. But this means exactly that $\gin(U)_d+x_1^2A_{d-2}=A_d$.
\end{proof}

\begin{cor}
\label{cor:ss_bound_independent_of_d}
If $k\le d$ then 
\[
m(n,d,k) \le m(n,k,k).
\]
\end{cor}
\begin{proof}
Let $U\subset A_d$ be a subspace of codimension $k$ such that $\codim U^2=m(n,d,k)$ and $k<d$. By \cref{thm:deg_reduction} we have $\codim U^2\le \codim V^2$ with $V=(U:l)$ for a generic linear form $l\in A_1$. By definition $\codim V^2\le m(n,d-1,k)$, hence we are done by induction.
\end{proof}

\begin{rem}
\begin{enumerate}[leftmargin=0.6cm]
\item[]
\item The reason we pass to initial ideals in the second part is because we need to show that
\[
\id{UV,l}_{2d-1}=A_{2d-1}.
\]
As we have seen $\id{U,l}_d=A_d$ and if we take another generic linear form $l'$ we also have $\id{V,l'}_{d-1}=A_{d-1}$. However since $V=(U:l)$ we do not know that $l$ behaves generically for $V$.
\item It is not true in general that $(U:l)$ is \bpf\ if $U$ is. Let $n=3$ and let $U=\spn(x^2y, x^2z, xy^2)^\perp\subset A_3$. Then $U$ contains $z\spn(z,y)A_1$ and thus for a generic linear form $l\in A_1$ we have $(U:l)=z\spn(y,z)\oplus\spn(p)$ for some $p\in A_2$. Hence the space $(U:l)$ has a \bp, namely $\sz(z,p)$.

One can show however that $(U:l)$ is \bpf\ whenever the degree is large enough.
\end{enumerate}
\end{rem}

\begin{example}
In fact, in all cases we know of $m(n,d,k)=m(n,k,k)$ for any $k\le d$. 

With SAGE one easily checks that $m(3,5,5)=16$ and $m(3,9,9)=31$. To show that $m(3,d,5)=m(3,5,5)$ and $m(3,d,9)=m(3,9,9)$ for $k\le d$ it is enough by \cref{cor:ss_bound_independent_of_d} to find some subspace in degree $d$ that realizes the bound $m(3,k,k)$. 

For $k=5$, let
\[
W=\spn(x_1^d,x_1^{d-1}x_2,x_1^{d-1}x_3,x_1^{d-2}x_2^2,x_1^{d-2}x_2x_3)
\]
and $U:=W^\perp$, then $\codim U^2=m(3,k,k)$. From the last corollary it therefore follows that $m(3,d,k)=m(3,k,k)$.

For $k=9$ we cannot take the nine smallest monomials (wrt the lex-ordering) to realize the maximum. For $d\ge 9$ let $s=d-9$ and let $U$ be the following subspace
\[
U^\perp=x_1^s\spn(x_1^{9}, x_1^{8}x_2, x_1^{8}x_3, x_1^{7}x_2^2, x_1^{7}x_2x_3, x_1^{7}x_3^2, x_1^{6}x_2^3, x_1^{6}x_2^2x_3, x_1^{5}x_2^4).
\]
Then it also follows that $m(3,d,9)=31$ for any $d\ge 9$.

This shows however that it is not clear in general which subspace realizes the maximum, even in the case $n=3$. We can for example not take lex-segment ideals which realize the bound in Macaulay's \cref{thm:macaulay} (i).
\end{example}

\section{Lifting subspaces}
\label{sec:Lifting subspaces}

In \cref{thm:reduction_number_of_variables} we showed how to reduce the number of variables, now we also want to increase that number while preserving $\codim U^2$.

\begin{dfn}
Let $U\subset A_d$ be a subspace of codimension $k$. Define
\[
U^{(1)}:=x_{n+1}A(n+1)_{d-1}\oplus U\subset A(n+1)_d\quad \text{(a subspace of codimension k)}
\]
and for any $l\ge 2$
\[
U^{(l)}:=(U^{(l-1)})^{(1)}\subset A(n+l)_d
\]
($U^{(0)}:=U$).  
\end{dfn}

\begin{prop}
\label{prop:liftingfaces}
Let $U\subset A_d$ be a subspace of codimension $k$. Let $H=(h_i)_{i\ge 0}$ be the Hilbert function of $\id{U}$. Then for every $l \ge 0$ the following hold:
\begin{enumerate}
\item The Hilbert function $K=(k_i)_{i\ge 0}$ of the ideal generated by $U^{(l)}$ in $A(n+l)$ satisfies
\begin{itemize}
\item $k_i=\dim A(n+l)_i$ for $0 \le i \le d-1$,
\item $k_i=h_i$ for $i\ge d$.
\end{itemize}
\item $\codim_{A(n+l)_{2d}} (U^{(l)})^2 = \codim_{A_{2d}} U^2 + l\cdot h_{2d-1}$.
\end{enumerate}
\end{prop}
\begin{proof}
It is enough to show this for $l=1$ since the rest follows by induction. Write $A'=A[y]$ with a new indeterminate $y$, then
\[
V:=U^{(1)}=yA'_{d-1}+U\subset A'_d.
\]
For any $s\ge 0$ we have
\begin{align*}
VA_{s}'&=yA_{d-1}'A_{s}'+UA_{s}'=\left(\bigoplus_{i=1}^{d+s} y^iA_{d+s-i}\right)+A_{s}U+yA_{s-1}U+\dots+y^sU\\
&=\bigoplus_{i=1}^{d+s} y^iA_{d+s-i} \oplus UA_{s}
\end{align*}
which shows (i) since $A'_{d+s}=\bigoplus_{i=0}^{d+s} y^iA_{d+s-i}$.

For (ii) we calculate $V^2$ and with the same argument we get
\[
V^2=y^2A_{2d-2}'+yA'_{d-1}U+U^2=\bigoplus_{i=2}^{2d} y^iA_{2d-i} \oplus y(A_{d-1}U) \oplus U^2
\]
and 
\[
\codim_{A'_{2d}} V^2 = \codim_{A_{2d}} U^2 + h_{2d-1}.
\]
\end{proof}

This enables us to determine the Hilbert function of codimension $2$ subspaces of $A_2$.

For generic $U$ the Hilbert function of $\id{U}$ will be as small as possible. In the codimension 2 case this means that the Hilbert function is $(1,n,2)$ generically. We will show that this holds whenever $U$ is \bpf.

\begin{prop}
\label{prop:quadratic_hf}
Let $U\subset A_2$ be a \bpf\ subspace of codimension 2. Then the Hilbert function of $\id{U}$ is $(1,n,2)$.
\end{prop}
\begin{proof}
By \cref{thm:macaulay} the Hilbert function is smaller or equal to $(1,n,2,2,\dots)$. So assume $h_{\id{U}}(3)>0$. Then by \cref{prop:liftingfaces} the subspace $U^{(l)}\subset A(n+l)_2$ has codimension 2 and for $l \ge 7$ we have $\codim_{A(n+l)_4} (U^{(l)})^2\ge 7$ which is not possible by \cref{thm:codimension2bounds}.
\end{proof}

\begin{prop}
\label{prop:lift_codim_1_2}
Let $U\subset A_d$ be a \bpf\ subspace of codimension $k\in\{1,2\}$ and $\codim U^2=s$, then for every $N\ge n$ there exists a \bpf\ subspace $V\subset A(N)_d$ of codimension $k$ such that $\codim V^2=s$.
\end{prop}
\begin{proof}
By \cref{prop:quadratic_hf} and \cref{cor:degree2d-1} the degree $2d-1$ component of $A/\id{U}$ has dimension $0$. Hence 
\[
\codim_{A(n)_{2d}} U^2=\codim_{A(n+l)_{2d}} (U^{(l)})^2.
\]
by \cref{prop:liftingfaces} (ii).
\end{proof}

\section{Arbitrary codimension}
\label{sec:arbitrary codimension}

We show bounds for $\codim U^2$ for \bpf\ subspaces $U$ of any codimension that are independent of $n$ and $d$ if $d$ is large enough.
The most important step is to also consider the orthogonal complement alongside our starting space, this is made precise in \cref{lem:dual}.

The main idea is the following: If $U\subset A_d$ is a \bpf\ subspace of codimension $k$, then consider $U':=U\cap A(m)_d$ for some $2\le m \le n$. To use \cref{thm:reduction_number_of_variables} we need to make sure that $\codim U=\codim U'$ and to get bounds that are independent of $n$ we want $U'$ to be \bpf\ as well (and $k\le d-1$), then $\codim U^2\le \codim (U')^2$.

We still always assume that $n,\,d\ge 2$ and $k\in\N$. 

\begin{rem}[The dual problem]
\label{rem:dual}
Let $U\subset A_d$ be a \bpf\ subspace of codimension $k$ and apply a generic \cc\ to $U$. Instead of asking if $U':= U\cap A(m)_d$ satisfies 
\begin{enumerate}
\item $\codim_{A(m)_d} U'=k$ and
\item $\sz(U')=\emptyset$ with $\sz(U')\subset\P^{m-1}$,
\end{enumerate}
as in \cref{thm:reduction_number_of_variables} and \cref{cor:reduction_number_of_variables}, we can also look at the dual problem:

Let $W=U^\perp$. Does $W':=W(x_1,\dots,x_m,0,\dots,0)$ have the same dimension as $W$ and does $W'$ not intersect the Veronese of $\P^{m-1}$?
\end{rem}

\begin{lem}
\label{lem:dual}
Let $U\subset A_d$ be a subspace and $W:=U^\perp$. Let $l_1,\dots,l_s\in A_1$ be linearly independent linear forms and $V:=\C[l_1,\dots,l_s]_d\subset A_d$. Then 
\[
(U\cap V)^\perp\cong(W+V^\perp)/V^\perp,
\]
and $V^\perp=\C[\lambda_1,\dots,\lambda_{n-s}]_d$ where $\spn(l_1,\dots,l_s)^\perp=\spn(\lambda_1,\dots,\lambda_{n-s})$.

Write $\ol W$ for $(W+V^\perp)/V^\perp$, then we especially have $\codim U\cap V=\dim \ol W$ and $U\cap V$ is \bpf\ if and only if $\ol W$ contains no $d$-th power of a linear form.
\end{lem}

Now we want to work on condition (i) in \cref{rem:dual} to ensure that $\dim W=\dim \ol W$.

Since $W$ will play the role of $U^\perp$, $k$ will usually denote the dimension of $W$, and not the codimension.

We start with a theorem that shows the geometric consequences of extremal behavior in Green's Theorem. This is used in the Proposition afterwards to determine the exact form of our subspaces.

\begin{thm}[{\cite[Theorem 3.2]{ams2018}}]
\label{thm:green_extremal}
Let $W\subset A_d$ be a subspace of dimension $k$ and suppose that for some $m,c\in\N$ we have
\[
\dim A_d-k=h_d=\binom{d+c}{d}+\binom{d+c-1}{d-1}+\dots+\binom{d+c-m}{d-m}.
\]
If $c_d=(h_d)_{(d)}|_0^{-1}$ then $W=\id{L_1,\dots,L_{n-c-2}}_d+FA_{d-m-1}$ where $L_1,\dots,L_{n-3}\in A_1$ are linearly independent linear forms and $F\in A_{m+1}$.
\end{thm}

This means that in this case $W$ is the saturated degree $d$ component of a hypersurface of degree $m+1$ in a $c+2$ dimensional subspace.

\begin{prop}
\label{prop:green_application_dimW}
Let $k\le n$ and let $W\subset A_d$ be a subspace of dimension $k$. Then for generic $l\in A_1$ either
\begin{enumerate}
\item $\dim \ol W=\dim W$, or
\item $\dim \ol W=\dim W - 1$ and $W=FA_1$ for some $F\in A_{d-1}$,
\end{enumerate}
where $\ol W=W+\id{l}_d/\id{l}_d$. Especially whenever $k<n$ it follows that $\dim \ol W=\dim W$.
\end{prop}
\begin{proof}
Using the notation of \cref{lab:setup_green} with $I=\id{W}$, we have 
\[
h_d=\dim A_d-k=\sum_{i=0}^{d-2} \underbrace{\binom{(n-2)+d-i}{d-i}}_{=\dim A(n-1)_{d-i}} + \binom{n-k}{1}.
\]
Green's theorem then tells us that 
\[
c_d\le \underbrace{\sum_{i=0}^{d-2} \binom{(n-2)+d-i-1}{d-i}}_{C:=} + \binom{n-k-1}{1}.
\]
As can be easily verified we have $C=\dim A(n-1)_d-(n-1)$. We then have
\[
c_d\le \dim A(n-1)_d-(n-1)+\binom{n-k-1}{1} = \dim A(n-1)_d - \begin{cases}k, \text{ if } k<n \\ k-1, \text{ if } k=n \end{cases}.
\]
Equivalently in the first case $\dim \ol W=k$ and in the second case $\dim \ol W\in\{k,k-1\}$. If $\dim \ol W=k-1$ then by \cref{thm:green_extremal} ($c=n-2,\, m=d-2$) it follows that $W=FA_1$ with $F\in A_{d-1}$.
\end{proof}

\begin{cor}
\label{cor:dim_intersection}
Let $U\subset A_d$ be a subspace of codimension $k$ with $k\le n$ then $\codim(U\cap \C[l_1,\dots,l_k]_d)=k$ for generic linear forms $l_1,\dots,l_k\in A_1$.

Especially after applying a generic \cc\ to $U$, we have $\codim (U\cap A(k)_d)=k$.
\end{cor}
\begin{proof}
Let $W=U^\perp$. If $k=n$ there is nothing to show, hence assume $k<n$. By \cref{lem:dual} it is enough to consider the dimension of $\ol W\subset A/\id{l_{k+1},\dots,l_n}$ for any basis $l_{k+1},\dots,l_n$ of the orthogonal complement of $\spn(l_1,\dots,l_k)$. Since $k<n$ it follows from \cref{prop:green_application_dimW} that $\dim W=\dim \ol W$.
\end{proof}

\begin{cor}
\label{cor:green_dim(W:l)}
Let $W\subset A_d$ be a subspace of dimension $k<n$ and $l\in A_1$ a generic linear form, then $\dim (W:l)=0$.
\end{cor}
\begin{proof}
This is immediate from \cref{prop:green_application_dimW} and the exact sequence in \cref{lab:setup_green}.
\end{proof}

This shows that condition (i) in \cref{rem:dual} is satisfied whenever $n>k$ and we go down by one variable. And it is in general not satisfied if $n\le k$ since we can take $W=FA_1$ for some $F\in A_{d-1}$.

Instead of using Green's Theorem we can also use generic initial ideals to show this.

\begin{prop}
\label{prop:gin_and_green}
Let $U\subset A_d$ be a subspace of codimension $k$ and $W:=U^\perp$. Let $V:=(\gin(U)_d)^\perp$. Then the number of monomials in $V$ divisible by $x_n$ is equal to $\dim (W:l)$ for generic $l\in A_1$.
\end{prop}
\begin{proof}
After a generic \cc\ we can assume that $\gin(U)=\initial(U)$, $V=\spn(T_1,\dots,T_k)$ for monomials $T_1,\dots,T_k$ and $l=x_n$. Let $s$ be the number of monomials in $V$ which are divisible by $x_n$ and assume that this is the case for $T_1,\dots,T_s$. Then $U$ has a basis given by
\[
\{M+\sum_{i=1}^k \lambda_M^i T_i\colon M\in \initial(U)_d \text{ a monomial},\ \lambda_M^i\in\C\}.
\]
We show that $U\cap A(n-1)_d$ has codimension $k-s$ in $A(n-1)_d$.

Let $M\in\initial(U)_d\cap A(n-1)_d$ be a monomial. Then $\lambda_M^i=0$ for $i\le s$ since $x_n>\dots>x_1$ and since we are working with the lex-ordering, hence $M+\sum_{i=1}^r \lambda_M^i T_i\in A(n-1)_d$. But the only monomials in $A(n-1)_d$ not contained in $\initial(U)_d$ are the ones in $V$ that are not divisible by $x_n$. These are $T_{s+1},\dots,T_k$. This shows that $\dim (U\cap A(n-1)_d)=\dim A(n-1)_d- (k-s)$.

But for $W$ this means that $\dim (W+\id{x_n}_d)/\id{x_n}_d = k-s$. Now the claim follows from the exact sequence in \cref{lab:setup_green}.
\end{proof}

\begin{rem}
Note that everything in \cref{prop:green_application_dimW} except for the form of $W$ in (ii) also instantly follows from \cref{prop:gin_and_green}: Since $\gin(U)$ is strongly stable, the maximal number of variables that can appear in $\gin(U)^\perp$ is achieved by $\gin(U)^\perp=x_1^{d-1}A(k)_1$. Hence if $n>k$, $x_n$ does not appear, if $n=k$ then either $x_n$ appears in exactly one monomial or it does not appear at all.
\end{rem}

Now we want to look at condition (ii) in \cref{rem:dual}. By \cref{lem:dual} asking whether $U':=U\cap A(n-1)$ is \bpf\ is the same as asking if the orthogonal complement contains no $d$-th power of a linear form. Thus assume that $U$ is \bpf\ and $W$ contains no $d$-th powers. Is it true that $\ol W\subset (A/\id{l})_d$ contains no $d$-th powers for generic $l\in A_1$ whenever $\dim W=\dim \ol W$? Sadly this is not true in general as the next example shows.

\begin{example}
Let $U:=(x_n^{d-1}A(n-1)_1)^\perp\subset A_d$ and let $l\in A_1$ be a generic linear form. After rescaling $l$ we can write $l=x_n+l'$ for some $l'\in A(n-1)_1$, hence $\ol W\cong (l')^{d-1}A(n-1)_1$. Then $(l')^d\in W$ and $\dim W=\dim \ol W$.
\end{example}

However it is true whenever the number of variables is large as the next theorem shows.

For convenience we use the following notation. For a subspace $W\subset A_d$ we say that $l\in A_1$ is $W$-generic if $l$ is generic in the sense of Green's \cref{thm:green_thm}.

\begin{prop}
\label{prop:stay_bpf}
Let $W\subset A_d$ be a subspace of dimension $k$ and $n\ge 3k+1$. If $W$ contains no $d$-th power of a linear form, then the same holds for $\ol W=W+\id{l}_d/\id{l}_d\subset (A/\id{l})_d$ where $l\in A_1$ is a generic linear form.
\end{prop}
\begin{proof}
Assume this is wrong and let $l\in A_1$ be $W$-generic. Since $\ol W$ contains a $d$-th power, there exist $L\in A_1$ and $g\in A_{d-1}$ such that $L^d+lg\in W$.

Let $s\le k+1$ be the largest integer such that there exist linearly independent linear forms $l_1,\dots,l_s$ that are $W$-generic and such that the $2s$ linear forms $l_1,\dots,l_s,L_1,\dots,L_s$ are linearly independent with $L_i^d+l_ig_i\in W$ ($L_1,\dots,L_s\in A_1$, $g_1,\dots,g_s\in A_{d-1}$).

If $s=k+1$ then after a \cc\ the elements $x_1^d+x_{k+2}g_1,\dots,x_{k+1}^d+x_{2k+2}g_{k+1}$ are contained in $W$ and are linearly dependent since $\dim W=k$. Let $\lambda_1,\dots,\lambda_{k+1}\in\C$ not all zero such that
\[
\sum_{i=1}^{k+1} \lambda_ix_i^d+\sum_{i=1}^{k+1} \lambda_ix_{k+1+i}p_i=0.
\]
For $1\le i\le k+1$ the element $x_i^d$ cannot appear in the second sum, hence both sums are zero. But looking at the first sum, it immediately follows that all $\lambda_i$ are zero.

If $s<k+1$, let $h_1,\dots,h_{k+1}$ be generic linear forms with $H_i^d+h_ip_i\in W$. By assumption we have $H_i\in\spn(l_1,\dots,l_s,L_1,\dots,L_s,h_i)$. By changing $p_i$ we can further assume that $H_i\in\spn(l_1,\dots,l_s,L_1,\dots,L_s)$ for every $i$. Since $\dim W=k$ these $k+1$ forms are linearly dependent and there exist $\lambda_1,\dots,\lambda_{k+1}\in\C$ not all zero such that
\[
\sum_{i=1}^{k+1} \lambda_iH_i^d+\sum_{i=1}^{k+1} \lambda_ih_ip_i=0.
\]
Since $h_1,\dots,h_{k+1}$ are generic and $n\ge 3k+1$, the linear forms $l_1,\dots,l_s,L_1,\dots,L_s,$ $h_1,\dots,h_{k+1}$ are linearly independent. Again after a \cc\ mapping $l_1,\dots,l_s,L_1,\dots,L_s$ to $x_1,\dots,x_{2s}$ and $h_1,\dots,h_{k+1}$ to $x_{2s+1},\dots,s_{2s+k+1}$, the first sum is contained in $\C[x_1,\dots,x_{2s}]$ and the second sum contains no monomial in the first $2s$ variables. It follows that both sums separately have to be zero. Especially
\[
\sum_{i=1}^{k+1} \lambda_ih_ip_i=0.
\]
By assumption not all $\lambda_i$ are zero. If only one was non-zero, say $\lambda_j\neq 0$, then it follows that $h_jp_j=0$, hence $H_j^d\in W$ which is a contradiction. Therefore we can assume that at least two of the $\lambda_i$ are non-zero. Let $j:=\max\{i\colon \lambda_i\neq 0\}$, then $h_jp_j\in\spn(h_1p_1,\dots,h_{j-1}p_{j-1}):= V\neq \{0\}$. This space $V$ does not depend on $h_j$, hence we can assume that $h_j$ is $V$-generic. We have $\dim V\le j-1\le k<n$ which means that $\dim (V:l)=0$ for a generic linear form $l$ by \cref{cor:green_dim(W:l)}. Especially this holds for $h_j$, a contradiction.
\end{proof}

\begin{thm}
\label{thm:independent_bound}
Let $k\le d-1$. Then for every $n\ge 2$ and every \bpf\ subspace $U\subset A(n)_d$ of codimension $k$ we have 
\[
\codim U^2\le m(3k,k,k).
\]
This constant is independent of $n$ and $d$.
\end{thm}
\begin{proof}
If $n\le 3k$ and $U\subset A_d$ is a \bpf\ subspace of codimension $k$ then by \cref{prop:liftingfaces} we have $\codim U^2=\codim (U^{(3k-n+1)})^2$. It is therefore enough to only consider the case $n > 3k$.

Apply a generic \cc\ to $U$, then by \cref{cor:dim_intersection} if follows that $V:=U\cap A(3k)_d$ has codimension $k$ in $A(3k)_d$ and by \cref{prop:stay_bpf} the subspace $V$ is still \bpf. 

It follows from \cref{cor:reduction_number_of_variables} that 
\[
\codim U^2\le \codim V^2 \le m(3k,d,k).
\]
By \cref{cor:ss_bound_independent_of_d} we have $m(3k,d,k)\le m(3k,k,k)$.
\end{proof}

\begin{rem}
\label{rem:growth_bound}
As we have noted earlier, we do not know the value or any reasonable bound for $m(3k,k,k)$. However it seems very likely that we have $m(3k,k,k)=2k^2+\frac{1}{6}(k^3+3k^2+2k)$ for all $k$: If $U$ is a (strongly stable) subspace such that $\codim U^2=m(3k,k,k)$, then we intersect with $A(k)_d$. With $U':=U\cap A(k)_d$ we then have $\codim U^2\le 2k^2+\codim (U')^2$ by \cref{cor:reduction_number_of_variables} and $\codim (U')^2\le m(k,k,k)$.
The difficulty is to find a bound for $m(k,k,k)$. It seems like $m(k,k,k)$ is always equal to $\frac{1}{6}(k^3+3k^2+2k)$ which is the value of $\codim U^2$ for $U=\spn(x_1^d,x_1^{d-1}x_2,\dots,x_1^{d-1}x_n)^\perp$. However, we can not prove that this holds.
\end{rem}

\begin{rem}
Assuming that $d$ is big enough is essential. Consider the following example: Let $R=A(4), \mm=\id{x_5,\dots,x_n}\subset A(n), n\ge 5$ and
\[
U=\spn(x_1^3,x_2^3,x_3^3,x_4^3,x_1^2x_2+x_3^2x_4)\subset A(4)_3.
\]
This subspace is generated by an almost complete intersection and is of codimension 15. The Hilbert function of $\id{U}$ is $(1,4,10,15,15,7,1)$. Define the subspace
\[
V=U^{(n-4)}=\bigoplus_{i=1}^3 \mm^i R_{3-i}\oplus U\subset A(n).
\]
This subspace also has codimension 15 in $A(n)$ and also has the "same" Hilbert function by \cref{prop:liftingfaces}. Again by \cref{prop:liftingfaces} we know that $\codim V^2=\codim U^2+7\cdot (n-4)$.

This shows that we cannot have a uniform bound for this combination of codimension and degree.
\end{rem}

It seems likely that one cannot only reduce to $3k$ variables in \cref{prop:stay_bpf} but actually to $k+1$ variables. This is at least the only counterexample we know of (for $k\le n-1$).

\begin{conj}
\label{conj:bpf_intersection}
Let $k\le d-1,n-1$ and $n\ge 3$. Let $W\subset A_d$ be a subspace of dimension $k$ and suppose that $W$ contains no $d$-th power of a linear form. Then for a generic linear form $l\in A_1$ it holds that either
\begin{enumerate}
\item $\ol W$ contains no power of a linear form, or
\item $n=k+1$ and $W=L_1^{d-1}\C[L_2,\dots,L_{k+1}]_1$ for some basis $L_1,\dots,L_{k+1}$ of $A(k+1)_1$.
\end{enumerate}
\end{conj}

This would allow us to show $\codim U^2\le m(k,k,k)$ in \cref{thm:independent_bound} with an additional argument.

\begin{rem}
\begin{enumerate}[leftmargin=0.6cm]
\item[]
\item Note that the conjecture is certainly false if $n=k$ is allowed: For $n=k$ let $W=x_1^{d-1}\C[x_2,\dots,x_n]_1\oplus \spn(p)$ for some generic $p\in A_d$. Since $p$ is generic $W$ contains no $d$-th powers, but $\ol W\cong x_1^{d-1}\C[x_1,\dots,x_{n-1}]_1\oplus\spn(\ol p)$ hence does contain one.
\item The conjecture is true for $n\ge 3k+1$ by \cref{prop:stay_bpf}. For $k=1$ it also follows from a simple geometric observation: If $W=\spn(p)$ and $p$ is not a power of a linear form, then $\sz(p)$ is non-degenerate. Hence the same holds for a generic hyperplane section and thus it cannot be a power of a linear form.
\end{enumerate}
\end{rem}

\begin{rem}
For faces of \gsa\ this can also be interpreted in the following way, comparing singular to \nsingular\ sos forms: Let $k\le d-1$. For any $n$ let $f\in\Sigma_{n,2d}$ be a \nsingular\ form with $F\subset\gram(f)$ a face of rank $\dim A_d-k$, and let $g\in\Sigma_{n,2d}$ be a (singular) form such that $\gram(g)$ has a face $F'$ with corresponding subspace $U=\mathcal{U}(F')$ that realizes $m(n,d,k)$.

If $n$ tends to infinity the same holds for $\dim F'-\dim F\ge m(n,d,k)-m(3k,d,k)$. So the dimensional differences between faces of a fixed rank between singular and \nsingular\ forms are arbitrarily large.

Thus when trying to understand generic \gsa\ it is essential to exclude singular forms, which in our situation means requiring subspaces to be \bpf.
\end{rem}

\bibliographystyle{plain}
\bibliography{mybib.bib}

\begin{thebibliography}{1}

\bibitem{ams2018}
Jeaman {Ahn}, Juan~C. {Migliore}, and Yong-Su {Shin}.
\newblock {Green's theorem and Gorenstein sequences.}
\newblock {\em {J. Pure Appl. Algebra}}, 222(2):387--413, 2018.

\bibitem{bc2018}
Mats {Boij} and Aldo {Conca}.
\newblock {On Fr\"oberg-Macaulay conjectures for algebras.}
\newblock {\em {Rend. Ist. Mat. Univ. Trieste}}, 50:139--147, 2018.

\bibitem{bh1998}
Winfried Bruns and H.~J\"urgen Herzog.
\newblock {\em Cohen-Macaulay Rings}.
\newblock Cambridge Studies in Advanced Mathematics. Cambridge University
  Press, 2 edition, 1998.

\bibitem{eisenbud1995}
David Eisenbud.
\newblock {\em Commutative Algebra: with a View Toward Algebraic Geometry}.
\newblock Springer New York, 1995.

\bibitem{green1989}
Mark Green.
\newblock Restrictions of linear series to hyperplanes, and some results of
  macaulay and gotzmann.
\newblock In {\em Algebraic Curves and Projective Geometry}, pages 76--86.
  Springer Berlin Heidelberg, 1989.

\bibitem{ika1999}
Anthony Iarrobino and Vassil Kanev.
\newblock {\em Power Sums, Gorenstein Algebras, and Determinantal Loci}.
\newblock Lecture Notes in Mathematics. Springer Berlin Heidelberg, 1 edition,
  1999.

\bibitem{ikl1999}
Anthony Iarrobino and Steve~L. Kleiman.
\newblock The gotzmann theorems and the hilbert scheme.
\newblock In {\em Power Sums, Gorenstein Algebras, and Determinantal Loci},
  pages 289--312. Springer Berlin Heidelberg, 1999.

\bibitem{scheiderer2018}
Claus Scheiderer.
\newblock Extreme points of gram spectrahedra of binary forms.
\newblock arXiv preprint arXiv:1802.05513, 2018.

\end{thebibliography}
\end{document}